\documentclass[12pt, reqno]{amsart}
\usepackage{amsmath, amsthm, amscd, amsfonts, amssymb, graphicx, color}
\usepackage[bookmarksnumbered, colorlinks, plainpages]{hyperref}

\newtheorem{theorem}{Theorem}[section]
\newtheorem{lemma}{Lemma}[section]
\newtheorem{proposition}{Proposition}[section]
\newtheorem{corollary}{Corollary}[section]
\theoremstyle{definition}
\newtheorem{definition}{Definition}[section]
\newtheorem{example}{Example}[section]

\theoremstyle{remark}
\newtheorem{remark}{Remark}[section]
\numberwithin{equation}{section}

\begin{document}
\title[\bf Some results on Higher orders quasi-isometries ]
{\bf  Some results on Higher orders quasi-isometries }
\author[ O. A. M. Sid Ahmed  A. Saddi, Khadija Gherairi]{ Sid Ahmed Ould Ahmed Mahmoud ,   Adel SADDI and Khadija Gherairi  }
\address{ Sid Ahmed Ould Ahmed Mahmoud  \endgraf
  Mathematics Department, College of Science, Jouf
University,\endgraf Sakaka P.O.Box 2014. Saudi Arabia}
\email{sidahmed@ju.edu.sa}
\address{
  Adel SADDI \endgraf
  Mathematics Department, College of Science,\endgraf Gabes
University. Tunisia}
 \email{saddi.adel@gmail.com}
\address{
 Khadija Gherairi  \endgraf
Laboratory of Mathematics and Applications, College of Sciencs,\endgraf Gabes
University. Tunisia}
\email{khadija.gherairi@gmail.com}
\begin{abstract}
The purpose of the present paper is to pursue further study of a
class of linear bounded operators, known as $n$-quasi-$m$-isometric operators acting
on an infinite complex separable Hilbert space ${\mathcal H}$. This generalizes the class of $m$-isometric operators on Hilbert space introduced by Agler and Stankus in \cite{AS}.
 The class of
$n$-quasi-$m$-isometric operators was defined by S. Mecheri and T. Prasad in \cite{MPR}
, where they have given some of their properties. Based, mainly, on  \cite{BMM}, \cite{BMN}, \cite{TMVN} and \cite{PD}, we contribute some other properties of such operators.
\end{abstract}
\date{\mbox{{\emph{November 3-2018}}}}
\maketitle

\section{Introduction}

\vspace*{1.5pt}
\quad  Throughout this paper $\mathbb{N}$  denotes the set of non negative integers, ${\mathcal H}$ stands for  an infinite separable complex Hilbert space with inner product
$\langle .\;| \;. \rangle$, ${\mathcal L}({\mathcal H})$ is the Banach algebra of all bounded linear operators on ${\mathcal H}$ and $I=I_{\mathcal H}$  the identity operator. For every $T\in {\mathcal B}({\mathcal H})$  we denote  by $\mathcal{R}(T)$, ${\mathcal N}(T)$ and $T^*$  the range, the null space and the adjoint of $T$  respectively. A  closed subspace ${\mathcal M}\subset {\mathcal H}$   is invariant for $T$ (or $T$-invariant) if $T( {\mathcal M})\subset {\mathcal M}$. As usual, the orthogonal complement and  the closure of ${\mathcal M}$ are denoted ${\mathcal M}^\perp$ and $\overline{\mathcal{M}}$ respectively. We denote by $P_\mathcal{M}$ the orthogonal projection on ${\mathcal M}$.\par \vskip 0.2 cm \noindent
Some of the most important subclasses of the algebra of all bounded linear operators
acting on a Hilbert space, are the classes of partial isometries  and quasi-isometries. An operator $T \in {\mathcal L}({\mathcal H})$
is said to be an isometry if $T^*T = I$, a partial isometry if $TT^*T = T$ and quasi-isometry if $T^{*2}T^2=T^*T.$ \par\vskip 0.2 cm \noindent In recent years
these classes has been generalized, in some sense, to the larger sets of operators so-called
$m$-isometries, $m$-partial isometries and $n$-quasi-isometries. An operator $T \in {\mathcal L}({\mathcal H})$ is said to be \par \vskip 0.2 cm \noindent $(1)$ $m$-isometric operator
for some integer $m \geq 1$ if it satisfies the operator equation
\begin{equation}\label{(1.1)}
\sum_{0\leq k\leq m}(-1)^{m-k}\binom{m}{k}T^{\ast k}T^k=0.
\end{equation}
It is immediate that $T$ is $m$-isometric operator if and only if
\begin{equation}\label{(1.2)}
\sum_{0\leq k\leq m}(-1)^{m-k}\binom{m}{k}\|T^kx\|^2=0\quad \forall\;x\in{\mathcal H},
\end{equation}
\noindent $(2)$
$m$-partial-isometry  for some  integer $m\geq 1$ if
\begin{eqnarray}\label{(1.3)}
T\bigg(\sum_{0\leq k\leq m}(-1)^k\binom{m}{k}T^{*m-k}T^{m-k}\bigg)=0,
\end{eqnarray}
\noindent$(3)$ $(m,q)$-partial isometry (or $q$-partial-$m$-isometry) if
\begin{eqnarray}\label{(1.4)}
T^q\bigg(\sum_{0\leq k\leq m}(-1)^k\binom{m}{k}T^{*m-k}T^{m-k}\bigg)=0,
\end{eqnarray}
$(4)$ \ $m$-quasi-isometry for some integer $n\geq 1$ if
\begin{eqnarray}\label{(1.5)}
T^{*m+1}T^{m+1}-T^{*m}T^m=0.
\end{eqnarray}
It is immediate that $T$ is $m$-quasi-isometric if and only if
$$ T^{*m}\bigg(T^*T-I\bigg)T^m=0.$$
Here $\displaystyle \binom{m}{k}$ is the binomial coefficient. In
\cite{AS}, J. Agler and M. Stankus initiated the study of operators $T$ that satisfy
the identity (\ref{(1.1)}). In  \cite{SA}, A. Saddi  and O. A. M. Sid Ahmed studied
operator $T$ which satisfies $(\ref{(1.3)})$. This concept was later generalized to the operators satisfying $(\ref{(1.4)})$,
 was defined  by O. A. M. Sid Ahmed \cite{OS1}. The study of operators satisfying $(\ref{(1.5)})$  was introduced and study by  L. Suciu in  \cite{LS}.
 The $1$-quasi-isometries are shortly called quasi-isometries, such
operators being firstly studied  in \cite{SP1} and \cite{SP2}.
. \par\vskip 0.2 cm \noindent
Recently, S. Mecheri and T. Prasad \cite{MPR} introduced the class of $n$-quasi-$m$-isometric
operators which generalizes the class of $m$-isometric operators and $n$-quasi-isometries. For positive integers $m$
and $n$, an operator $T \in {\mathcal L}({\mathcal H})$ is said to be an $n$-quasi-$m$-isometric operator if
\begin{eqnarray}\label{(1.6)}
T^{*n}\bigg(\sum_{0\leq k\leq m}(-1)^k\binom{m}{k}T^{*m-k}T^{m-k}\bigg)T^n=0,
\end{eqnarray}
After an introduction on the subject and some connection with known
facts in this context, the results of the paper are briefly described. In section two, we give a matrix characterization
of $n$-quasi-$m$-isometries by using the decomposition  ${\mathcal H}=\overline{T^n({\mathcal H})} \oplus {T^{*}}^{-n}(0)$.  Several properties are proved by exploiting the special kind of operator matrix representation associated with such operators. In the course of our investigation, we find  some
properties of $m$-isometries  which are retained by $n$-quasi-$m$-isometries.
In particular, we show that if $ T\in {\mathcal L}({\mathcal H})$ is an $n$-quasi-isometry then its power is an $n$-quasi-isometry. If  $T$ and $S$  are doubly commuting such that $T$ is an $n_1$-quasi-$m$-isometry and $S$ is an $n_2$-quasi-$l$-isometry, then
$TS$ is a $n_0 = \max\{n_1, n_2\}$-quasi-$(m +l-1)$-isometry. It has also been proved  that the sum of an $n$-quasi-$m$-isometry and a commuting
nilpotent operator of degree $p$ is a $2\max\{n, p\}$-quasi-$(m+2p-2)$-isometry.
In section three, we recall the definition of $n$-quasi strict-$m$-isometries and
we give some of their properties which are  similar to those of $n$-quasi-$m$-isometries.

\section{Some properties of $n$-quasi-$m$-isometric operators}
\par\vskip 0.2 cm \noindent
 In this section, we study some further  properties of $n$-quasi-$m$-isometries. First, we will start with the following notations. \par\vskip 0.2 cm \noindent
 For $T\in {\mathcal L}({\mathcal H})$, we set
\vspace{-2mm}
\begin{equation} \label{2.1}
\beta_m(T):=\sum_{0\leq k\leq m}(-1)^{m-k}\binom{m}{k}T^{\ast k}T^k,\end{equation}
\begin{equation}\label{2.2}
\beta_{m,\;n}(T):=T^{\ast n}\bigg(\sum_{0\leq k\leq m}(-1)^{m-k}\binom{m}{k}T^{\ast k}T^k\bigg)T^n
\end{equation}
and
\begin{equation}\label{2.3}
\Delta_{m,\;n}(T,x):=\sum_{0\leq k \leq m}(-1)^{m-k}\binom{m}{k}\|T^{k+n}x\|^2,\quad \;\;\;\;x\in{\mathcal H}.
\end{equation}
Observe that $T$ is an $n$-quasi-$m$-isometric operator if  and only if $\beta_{m,\;n}(T)=0$ or equivalently if
\begin{equation*}\label{2.4}
\Delta_{m,\;n}(T,x)=0\quad \forall\;\;x\in{\mathcal H}.
\end{equation*}
\vspace {0.5 mm}
\begin{lemma}\label{lem21}
Let $T\in {\mathcal L}({\mathcal H})$, then $T$ is an $n$-quasi-$m$-isometric operator if and only if
\begin{equation}
\sum_{0\leq k\leq m}(-1)^{m-k}\binom{m}{k}\|T^{k}x\|^2=0\quad \forall\;\;x\in \overline{{\mathcal R}(T^n)},
\end{equation}
or $T$ is an $m$-isometric operator on  $\overline{{\mathcal R}(T^n)}.$
\end{lemma}
\begin{proof}
Obvious.
\end{proof}
\vspace {0.5 mm}
 Let $\mathbb{Z}$ denote
the set of integers and $\mathbb{Z}_+$ denote the set of nonnegative integers.
\vspace {0.5 mm}
\begin{lemma}\label{lem22}(\cite[Lemma 5.4]{GU})
 Let $\big(a_j\big)_{j\in \mathbb{Z}_+}$
be a sequence of real numbers. Then
$$\sum_{0\leq k \leq m}(-1)^{m-k}\binom{m}{k}a_{j+k}=0\;\;\;\text{for}\;\;j\geq 0$$
if and only if there exists a polynomial $P$ of degree less than or equal to $m-1$
such that $a_j = P(j)$. In this case $P$ is the unique polynomial interpolating
$\{(j, a_j)\}$, $0 \leq j\leq m-1.$
\end{lemma}
\vspace {0.5 mm}
\begin{proposition}\label{prop21}
Let $T\in {\mathcal L}({\mathcal H})$ and $x\in {\mathcal H}$. We set   $a_{k+n}:= \|T^{n+k}x\|^2$. Then
$T$ is an $n$-quasi-$m$-isometry if and only if for each  $ x\in {\mathcal H}$, there exists a polynomial
$P$   of degree less than or equal to $m-1$, such that  $a_k = P(k)$ for  $n \in \mathbb{Z}_+$.
\end{proposition}
\begin{proof}
The proof is a consequence of  Lemma \ref{lem21}, Lemma \ref{lem22} and \cite[Theorem 5.5]{GU}.
\end{proof}

In \cite{MPR}, S. Mecheri and T. Prasad studied the matrix representation of $n$-quasi-$m$-isometric  operator
with respect to the direct sum of $\overline{{\mathcal R}(T^n)}$ and its orthogonal complement. In the
following we give an equivalent condition for $T$ to be $n$-quasi-$m$-isometric operator. Using this result
we obtained several important properties  of this class of operators.
\vspace {-0.4 mm}
\begin{theorem}\label{th21}
Let $T\in {\mathcal L}({\mathcal H})$ such that $\overline{{\mathcal R}(T^n)} \not={\mathcal H}$, then the following statements are equivalent.\par \vskip 0.2 cm \noindent $(1)$ $T$ is an $n$-quasi-$m$-isometric operator.\par \vskip 0.2 cm \noindent $(2)$ $T=\left(
      \begin{array}{cc}
        T_1 & T_2 \\
        0 & T_3 \\
      \end{array}
    \right)$ on ${\mathcal H}=\overline{{\mathcal R}(T^n)}\oplus {\mathcal N}(T^{*n})$, where $T_1$ is an $m$-isometric operator and
    $T_3^n=0$.
\end{theorem}
\begin{proof} $(1)\Rightarrow (2)$,  follows  by \cite[Lemma 2.1]{MPR}.\par \vskip 0.2 cm \noindent
\noindent $(2)\Rightarrow (1)$ Suppose that $T=\left(
      \begin{array}{cc}
        T_1 & T_2 \\
        0 & T_3 \\
      \end{array}
    \right)$ onto ${\mathcal H}=\overline{{\mathcal R}(T^n)}\oplus  {\mathcal N}(T^{*n})$
    , with $$ \beta_m(T_1):=\sum_{0\leq k\leq m}(-1)^{m-k}\binom{m}{k}T_1^{*k}T_1^k=0 \;\;\text{and}\;\;T_3^n=0.$$
    Since
$ T^k= \left(
      \begin{array}{cc}
        T_1 ^k& \displaystyle\sum_{0\leq  j\leq k-1}T_1^jT_2T_3^{k-1-j} \\
        0 & T_3^k \\
      \end{array}
    \right)$ we have
    \begin{eqnarray*}
  && T^{*n}\bigg(\sum_{0\leq k \leq m}(-1)^{m-k}\binom{m}{k}T^{*k}T^k\bigg)T^n\\&=&
  \left(
      \begin{array}{cc}
        T_1 & T_2 \\
        0 & T_3 \\
      \end{array}
    \right)^{*n}\bigg\{ (-1)^mI+\sum_{1\leq k\leq m}(-1)^{m-k}\binom{m}{k}\left(
      \begin{array}{cc}
        T_1 & T_2 \\
        0 & T_3 \\
      \end{array}
    \right)^{*k}\left(
      \begin{array}{cc}
        T_1 & T_2 \\
        0 & T_3 \\
      \end{array}
    \right)^k\bigg\}\\&&\times \left(
      \begin{array}{cc}
        T_1 & T_2 \\
        0 & T_3 \\
      \end{array}
    \right)^n \\&=&
    \left(
      \begin{array}{cc}
        T_1 ^{*n}& 0 \\
        \displaystyle\sum_{0\leq j\leq n-1}T_3^{*n-1-j}T_2^*T_1^{*j} & T_3^{*n} \\
      \end{array}
    \right)\\&&\times \bigg\{ (-1)^mI+\sum_{1\leq k\leq m}(-1)^k\binom{m}{k} \left( \begin{array}{cc}
        T_1 ^{*k}& 0 \\
        \displaystyle\sum_{ 0\leq j\leq k-1}T_3^{*k-1-j}T_2^*T_1^{*j}C_1 & T_3^{*k} \\
      \end{array}
    \right)\left(
      \begin{array}{cc}
        T_1 ^k& \displaystyle\sum_{j=0}^{ k-1}T_1^jT_2T_3^{k-1-j} \\
        0 & T_3^k \\
      \end{array}
    \right)\bigg\}\\&&\times \left(
      \begin{array}{cc}
        T_1 ^n& \displaystyle\sum_{j=0}^{n-1}T_1^jT_2T_3^{n-1-j} \\
        0 & T_3^n \\
      \end{array}
    \right)\\ \end{eqnarray*} \begin{eqnarray*}&=& \left(
      \begin{array}{cc}
        T_1 ^{*n}& 0 \\
        \displaystyle\sum_{0\leq j\leq n-1}T_3^{*n-1-j}T_2^*T_1^{*j} & 0 \\
      \end{array}
    \right)\\&&\times \bigg\{  \left( \begin{array}{cc}
        \Lambda_m(T_1)& C\\
        \\
        D &   B    \\
      \end{array}
    \right)  \bigg\}\times\left(
      \begin{array}{cc}
        T_1 ^n& \displaystyle\sum_{ j=0}^{ n-1}T_1^jT_2T_3^{n-1-j} \\
        0 & 0 \\
      \end{array}
    \right)\\
    %\end{eqnarray*}
    %\begin{eqnarray*}
    % T^{*n}\bigg(\sum_{0\leq k \leq m}(-1)^{m-k}\binom{m}{k}T^{*k}T^k\bigg)T^n\
    &=&\left(
           \begin{array}{cc}
             T_1^{\ast n}  \Lambda_m(T_1)T_1^n& 0 \\
             \\
               0&   0 \\
           \end{array}
         \right)\\& =& 0.
    \end{eqnarray*}
    Therefore  $\beta_{m,\;n}(T)=0.$ Thus $T$ is an $n$-quasi-$m$-isometric operator.
\end{proof}
\par\vskip 0.2 cm \noindent
For $T \in {\mathcal L}({\mathcal H})$, we denote by  $\sigma(T)$, $\sigma_{ap}(T)$ and $\sigma_p(T)$ respectively  the spectrum, the approximate point spectrum and the point spectrum of $T$.
\par \vskip 0.2 cm \noindent
\begin{corollary}\label{cor21}
$T\in {\mathcal L}({\mathcal H})$ be an $n$-quasi-$m$-isometric operator. The following statements hold.
 \begin{itemize}
 \item[(i)]\;$\sigma(T)=\sigma(T_1)\cup\{0\}$ where $T_1=T_{/\overline{{\mathcal R}(T^n)}}$.
  \item[(ii)]\; $T_1$ is bounded below.
  \item[(iii)]\; If $\mu\in\sigma_{ap}(T)\setminus\{0\}$ then ${\overline{\mu}}\in\sigma_{ap}(T^*)$. In particular, if  $\mu\in\sigma_{p}(T)\setminus\{0\}$ then ${\overline{\mu}}\in\sigma_{p}(T^*)$.

 \end{itemize}
  \end{corollary}
  \begin{proof}
  $(i)$\; Since $T$ is an $n$-quasi-$m$-isometric operator, it follows by Theorem \ref{th21} that
  $$ T=\left(
         \begin{array}{ccc}
           T_1 & T_2 \\
          0 & T_3
         \end{array}
       \right)\;\;\text{on}\;\;{\mathcal H} =\overline{{\mathcal R}(T^n)}\oplus {\mathcal N}(T^{*n}),$$ where $T_1$ is an $m$-isometric operator and $T_3^n=0$.
   From \cite[Corollary 7]{JHW}, it follows that $\sigma(T)\cup W=\sigma(T_1)\cup\sigma(T_3)$, where $W$ is the union of certain of the holes in $\sigma(T)$   which is a subset
of $\sigma(T_1)\cap \sigma(T_3)$.  Further $\sigma{(T_3)}=\{0\}$  and $\sigma(T_1)\cap \sigma(T_3)$ has no
interior points. So we have  by \cite[Corollary 8]{JHW}
$$\sigma(T) = \sigma(T_1) \cup \sigma(T_3) = \sigma(T_1) \cup \{0\}.$$\par \vskip 0.1 cm \noindent
$(ii)$ By \cite[Lemma 1.21]{AS}, it is well known that the approximate spectrum of $T_1$ lies in unit circle. Hence $0\notin \sigma_{ap}(T_1)$. Consequently, $T_1$ is bounded from below.\par\vskip 0.2 cm \noindent
$(iii)$ The proof  follows from  \cite [ Theorem 2.2]{AS}.
  \end{proof}
\medskip
\noindent Recall that two operators $T\in {\mathcal L}({\mathcal H})$ and  $S\in {\mathcal L}({\mathcal H})$ are similar if there exists an invertible operator   $X\in {\mathcal L}({\mathcal H})$
 such that $XT=SX$ (i.e; $T=X^{-1}SX$  or $S=XTX^{-1}$).
\begin{corollary}\label{cor22} Let $T\in {\mathcal L}({\mathcal H})$ be an $n$-quasi-$m$-isometric operator. If $T_1=T_{/\overline{{\mathcal R}(T^n)}}$ is invertible, then
$T$ is similar to a direct sum of a $m$-isometric operator and a nilpotent operator.
       \end{corollary}
       \begin{proof}
       By Theorem \ref{th21} we write the matrix representation of $T$ on ${\mathcal H} =\overline{{\mathcal R}(T^n)}\oplus {\mathcal N} (T^{*n})$ as follows $T=\left(
         \begin{array}{ccc}
           T_1 & T_2 \\
          0 & T_3
         \end{array}
       \right)$ where $T_1=T_{/\overline{{\mathcal R}(T^n)}}$ is an $m$-isometric operator and $T_3^n=0$.
       Since $T_1$ is invertible, we have $\sigma(T_1)\cap \sigma(T_3) = \emptyset$. Then there exists an operator $A$ such that $T_1A-AT_3 = T_2$ by \cite{Ro}. Hence
$$T=\left(
         \begin{array}{ccc}
           T_1 & T_2 \\
          0 & T_3
         \end{array}
       \right)=\left(
         \begin{array}{ccc}
           I & A \\
          0 & I
         \end{array}
       \right)^{-1}\left(
         \begin{array}{ccc}
           T_1 & 0 \\
          0 & T_3
         \end{array}
       \right)\left(
         \begin{array}{ccc}
           I & A \\
          0 & I
         \end{array}
       \right).$$  The desired result follows from Theorem \ref{th21}.
       \end{proof}
       \par \vskip 0.2 cm \noindent
       \par \vskip 0.2 cm \noindent Clearly that every $n$-quasi-$m$-isometric operator is an $(n+1)$-quasi-$m$-isometric operator. In \cite[Theorem 2.4]{SP}, the authors S. Mechri and S. M. Patel proved that if $T$ is a quasi-$2$-isometry, then $T$ is quasi-$m$-isometry for all $m \geq 2.$
       In the following corollary, we give a generalization that
 every $n$-quasi-$m$-isometric operator is an $n$-quasi-$k$-isometric operator for $k\geq m$.
\begin{corollary}\label{cor23}
Let $T\in {\mathcal L}({\mathcal H})$. If $T$ is an $n$-quasi-$m$-isometric operator, then $T$ is an $n$-quasi-$k$-isometric operator for every positive integer $k\geq m$.
\end{corollary}
\begin{proof}
If ${\mathcal R}(T^n)$ is dense, then $T$ is an $m$-isometric operator. Hence $T$ is a $k$-isometric operator for every positive integer $k\geq m$.\par \vskip 0.2 cm \noindent If ${\mathcal R}(T^n)$ is not dense, by Theorem \ref{th21} we write the matrix representation of T on \\${\mathcal H} =\overline{{\mathcal R}(T^n)}\oplus {\mathcal N} (T^{*n})$ as follows $T=\left(
         \begin{array}{ccc}
           T_1 & T_2 \\
          0 & T_3
         \end{array}
       \right)$ where $T_1=T_{/\overline{{\mathcal R}(T^n)}}$ is an $m$-isometric operator and $T_3^n=0$. Obviously, $T_1$ is a $k$-isometric operator for every integer $k\geq m$. The conclusion follows from the statement 2. of Theorem \ref{th21}.

\end{proof}
%\end{document}
\par \vskip 0.02 cm \noindent
We consider the following example of $n$-quasi-$m$-isometry map, which is not a quasi-$m$-isometry.\par \vskip 0.2 cm \noindent
\begin{example}
Let $\big({e_k}\big)_{k\in {\mathbb{ N}}}$
 be an orthonormal basis of  ${\mathcal H}$. Define $T\in {\mathcal L}({\mathcal H})$ as follows
 $$Te_1=2e_2,\;Te_2=3e_3\;\;\;\;\text{and}\;\;\;T_k=e_{k+1}\;\;\;\;\text{for}\;\;k\geq 3.$$ Then by a straightforward calculation,
one can show that $T$ is a $2$-quasi-$2$-isometry but it is not a quasi-$2$-isometry.
\end{example}
\par \vskip 0.2 cm \noindent
In the following theorem, we  give a sufficient condition on an
$n$-quasi-$m$-isometric operator for $n\geq 2$ to be a quasi-$m$-isometric operator.

\begin{theorem}\label{2.2}
Let $T\in \mathcal{L}(\mathcal{H})$ be an $n$-quasi-$m$-isometry  for $n\geq 2$. If $\mathcal{N}(T^*)=\mathcal{N}(T^{*2})$, then $T$ is a quasi-$m$-isometry.
\end{theorem}

\begin{proof} Two different proofs will be given.\par \vskip 0.2 cm \noindent
 First proof. Under the assumption  $\mathcal{N}(T^*)=\mathcal{N}(T^{*2})$, it follows  that
$\mathcal{N}(T^*)=\mathcal{N}(T^{*n})$.
Since $T$ is an $n$-quasi-$m$-isometry, we have \begin{eqnarray*} T^{*n}\Bigg(\sum_{0\leq k \leq
m}(-1)^{m-k}\binom{m}{k}T^{*k}T^{k} \Bigg)T^n=0,
\end{eqnarray*}
we deduce that  \begin{eqnarray*} T^*\Bigg( \sum_{0\leq k \leq
m}(-1)^{m-k}\binom{m}{k}T^{*k}T^{k}\Bigg)T^n=0.
\end{eqnarray*}
This means that
\begin{eqnarray*} T^{*n}\Bigg( \sum_{0\leq k \leq
m}(-1)^{m-k}\binom{m}{k}\bigg(T^{*k}T^{k}\bigg)^*\Bigg)T=0.
\end{eqnarray*}
Using again the condition ${\mathcal N}(T^*)={\mathcal N}(T^{*n})$,  we obtain
\begin{eqnarray*} T^{*}\Bigg( \sum_{0\leq k \leq
m}(-1)^{m-k}\binom{m}{k}\bigg(T^{*k}T^{k}\bigg)^*\Bigg)T=0.
\end{eqnarray*}
Thus we have
\begin{eqnarray*} T^{*}\Bigg( \sum_{0\leq k \leq
m}(-1)^{m-k}\binom{m}{k} T^{*k}T^{k} \Bigg)T=0.
\end{eqnarray*}
Hence $T$ is a quasi-$m$-isometric operator.\par \vskip 0.2 cm \noindent
 Second proof. By the assumption that ${\mathcal N}(T^*)={\mathcal N}(T^{*2})$, we have
${\mathcal N}(T^*)={\mathcal N}(T^{*n})$ and therefore $\overline{{\mathcal R}(T)}=\overline{{\mathcal R}(T^n)}$. Now  Since $T$ is an $n$-quasi-$m$-isometry, it follows in view of Lemma 2.1 that $T$ is an $m$-isometry on $\overline{{\mathcal R}(T^n)}=\overline{{\mathcal R}(T)}$. This means that $T$ is an  $m$-isometry on $\overline{{\mathcal R}(T)}$. Again by applying Lemma \ref{lem21}, we obtain that $T$ is an quasi-$m$-isometry as required.
\end{proof}
\par \vskip 0.2 cm \noindent \begin{remark}
The following example shows that Theorem \ref{2.2} is not necessarily true if ${\mathcal N}(T^*)\not={\mathcal N}(T^{*2})$.\end{remark}
\begin{example}
Consider the  operator $T=\left(
                              \begin{array}{cc}
                                0 & 1 \\
                                0 & 0 \\
                              \end{array}
                            \right)$
 acting on the two dimensional Hilbert space $\mathbb{C}^2$. Then by a straightforward calculation,
one can show that $T$ is a $2$-quasi-isometry but it is not a quasi-isometry. However ${\mathcal N}(T^*)\not={\mathcal N}(T^{*2})$
\end{example}
\par \vskip 0.2 cm \noindent
Patel \cite[ Theorem 2.1]{MP}, proved that any power of a $(2, 2)$-isometry is again a
$(2,2)$-isometry. T. Berm\'{u}ndez et al. \cite[Theorem 3.1]{BMM}  proved that any power of $(m, p)$-isometry is an $(m, p)$-isometry.
 Later S. Mechri and S. M. Patel \cite{SP} gave a partial generalization to quasi-$2$-isometry.
The following theorem shows that any power of an $n$-quasi-$m$-isometry
is an $n$-quasi-$m$-isometry.\par\vskip 0.2 cm \noindent
\begin{theorem}\label{2.3}
Let $T \in {\mathcal{L}}({\mathcal H})$. Suppose $T$  is an $n$-quasi-$m$-isometric operator. Then $T^k$ is $n$-quasi-$m$-isometric operator for any $k\in\mathbb{N}$.
\end{theorem}
\begin{proof}
\par \vskip 0.1 cm \noindent Two different proofs of this statement will be given.\par \vskip 0.2 cm \noindent
 First proof. Suppose that $T$ is $n$-quasi-$m$-isometric operator. By the statement in Lemma \ref{2.1}, $T$ is $m$-isometric on $\overline{{\mathcal R}(T^n)}$. Therefore, in view of
\cite[Theorem 3.1]{BMM}, the operator $T^k$ is an $m$-isometric on   $\overline{{\mathcal R}(T^n)}$.  Thus
$$<\beta_m(T^k) x\,| \, x> \;= 0, \; \; \forall x\in   \overline{{\mathcal R}(T^n)}.$$
Using the inclusion   $$\overline{{\mathcal R}((T^k)^n)} \subset  \overline{{\mathcal R}(T^n)},$$
we get
$$ <\beta_m(T^k) x\,| \, x>\; = 0, \; \; \forall x\in \overline{{\mathcal R}((T^k)^n)}.$$
Hence $T^k$ is a $m$-isometric  on $\overline{{\mathcal R}((T^k)^n)}$. This shows, by the statement of Lemma \ref{2.1}, that
$T^k$ is $n$-quasi-$m$-isometric.\par \vskip 0.2 cm \noindent Second proof. If ${\mathcal R}(T^n)$ is dense, then $T$ is an $m$-isometric operator and hence $T^k$ is an $m$-isometric operator (by \cite[Theorem 3.1]{BMM}). If ${\mathcal R}(T^n)$ is not dense, by Theorem \ref{2.1} we write the matrix representation of $T$ on ${\mathcal H} =\overline{{\mathcal R}(T^n)} \oplus {\mathcal N}(T^{*n})$  as follows $T=\left(
         \begin{array}{ccc}
           T_1 & T_2 \\
          0 & T_3
         \end{array}
       \right)$ where $T_1 = T_{| \overline{{\mathcal R}(T ^n)}}$ is an $m$-isometric  and $ T_3 ^n=0.$ We notice that $$T^k=\left(
         \begin{array}{ccc}
           T_1^k & \displaystyle\sum_{ j=0}^{ k-1}T_1^jT_2T_3^{k-1-j} \\
          0 & T_3^k
         \end{array}
       \right),$$ where $T_1^k$ is an $m$-isometric  and $(T_3^k)^n=0.$  Hence $T^k$ is an $n$-quasi-$m$-isometric operator by Theorem \ref{2.1}.
\end{proof}
        \par\vskip 0.2 cm \noindent
       \begin{remark}
       The converse of Theorem \ref{2.3} in not true in general as it shows in the following example.
       \end{remark}
       \vspace{-4mm}
       \begin{example} It is not difficult to prove that the operator  $ T :=\left(
                                                                 \begin{array}{cc}
                                                                   -1 & -1 \\
                                                                   3& 2 \\
                                                                 \end{array}
                                                               \right)$
defined in $\mathbb{C}^2$
with the euclidean norm satisfies  $T^3$ is a quasi-$3$-isometry but $T$ is not a quasi-$3$-isometry.

       \end{example}

       The following theorem generalize \cite[Theorem 3.6]{BMM}.
\begin{theorem}\label{2.4}
Let $T\in {\mathcal L}({\mathcal H})$ and $r, s, m, n, l$ be
positive integers. If $T^r$ is an $n$-quasi-$m$-isometry and
$T^s$ is an $n$-quasi-$l$-isometry, then $T^q$ is an  $n$-quasi-$p$-isometry, where $q$ is the greatest
common divisor of $r$ and $s$, and $p$ is the minimum of  $m$ and  $l$.
\end{theorem}
\begin{proof} Consider the matrix representation of T with respect to the decomposition \\$ {\mathcal H}=\overline{{\mathcal R}(T^n)}\oplus {\mathcal N}(T^{*n})$
  as follows  $T=\left(
         \begin{array}{ccc}
           T_1 & T_2 \\
          0 & T_3
         \end{array}
       \right)$   where $T_1=T_{/ \overline{{\mathcal R}(T^n)}}$. \par \vskip 0.2 cm \noindent We have $ T^r= \left(
      \begin{array}{cc}
        T_1 ^r& \displaystyle\sum_{ j=0}^{ r-1}T_1^jT_2T_3^{r-1-j} \\
        0 & T_3^r \\
      \end{array}
    \right)$. Since $T^r$ is an $n$-quasi-$m$-isometry, we need to prove that $T_1^r$ is an $m$-isometry and $\big(T_3^r\big)^n=0$.\par \vskip 0.2 cm \noindent
    In fact, Let $P=P_{\overline{{\mathcal R}(T^n)}}$ be the projection onto $\overline{{\mathcal R}(T^n)}$. Then $$\left(
      \begin{array}{cc}
        T_1^r & 0 \\
        0 & 0\\
      \end{array}
    \right)=T^rP=PT^rP.$$\par \vskip 0.2 cm \noindent  Since $T^r$ is an $n$-quasi-$m$-isometry, then we have
    $$P\bigg(\sum_{0\leq k\leq m}(-1)^{m-k}\binom{m}{k} \big(T^r\big)^{\ast k}\big(T^r\big)^k\bigg)P= 0.$$
    That is
    $$\sum_{0\leq k\leq m}(-1)^{m-k}\binom{m}{k} \big(T_1^r\big)^{\ast k}\big(T_1^r\big)^k= 0.$$ Hence, $T_1$ is an $m$-isometry.
\par \vskip 0.1 cm \noindent
On the other hand, let $x=x_1+x_2\in  {\mathcal H}=\overline{{\mathcal R}(T^n)}\oplus  {\mathcal N}(T^{*n}).$  A
simple computation shows that
\begin{eqnarray*}
\langle \big(T_3^r\big)^{n}x_2,x_2\rangle &=& \langle \big(T^r\big)^{n}(I-P)x,(I-P)x \rangle\\
&=&\langle (I-P)x, \big(T^r\big)^{\ast n}(I-P)x \rangle=0.
\end{eqnarray*}
So, $\big(T_3^r\big)^n=0.$\par \vskip 0.2 cm \noindent Analogously, as $ T^s= \left(
      \begin{array}{cc}
        T_1 ^s& \displaystyle\sum_{ j=0}^{ s-1}T_1^jT_2T_3^{s-1-j} \\
        0 & T_3^s \\
      \end{array}
    \right)$  is an $n$-quasi-$l$-isometry by similar arguments we can conclude that $T_1^s$ is an $l$-isometry and $\big(T_3^s\big)^n=0.$ \par \vskip 0.2 cm \noindent Now, we  obtained that $T_1^r$ is an $m$-isometry and $T_1^s$ is an $l$-isometry. By \cite[Theorem 3.6]{BMM}, it follows that $T_1^q$ is a $p=\min(m,l)$-isometry. Moreover we have $\big(T_3^q\big)^n=0$.\par \vskip 0.2 cm \noindent Consequently, $ T^q= \left(
    \begin{array}{cc}
        T_1 ^q& \displaystyle\sum_{ j=0}^{ q-1}T_1^jT_2T_3^{q-1-j} \\
        0 & T_3^q \\
      \end{array}
    \right)$  is an $n$-quasi-$p$-isometry by Theorem \ref{2.1}. The proof is completed.

\end{proof}

\par \vskip 0.2 cm \noindent
The following corollary shows that
if we assume that two suitable different powers of $T$ are $n$-quasi-$m$-isometries,
then we obtain that $T$ is a $n$-quasi-$m$-isometry.
\par\vskip 0.2 cm \noindent
\begin{corollary}\label{2.4}
Let $T \in {\mathcal L}({\mathcal H})$ and  $r, s,m,n,l$ be
positive integers. The following properties hold. \par \vskip 0.2 cm \noindent $(1)$
If $T$ is an $n$-quasi-$m$-isometry such that $T^s$ is an $n$-quasi-isometry, then $T$ is

 an $n$-quasi-isometry.\par \vskip 0.2 cm \noindent $(2)$
 If $T^r$ and $T^{r+1}$ are $n$-quasi-$m$-isometries, then so is $T$.\par \vskip 0.2 cm \noindent $(3)$
 If $T^r$ is an $n$-quasi-$m$-isometry and $T^{r+1}$ is an $n$-quasi-$l$-isometry with
$m < l$,

 then $T$ is an $n$-quasi-$m$-isometry.
\end{corollary}
\begin{proof}

 The proof is an immediate consequence of Theorem \ref{2.4}.
\end{proof}
\medskip
Recall that an operator $T\in {\mathcal L}({\mathcal H})$ is said to be power bounded, if $\displaystyle\sup_{k}\|T^k\|< \infty$  or equivalently, there exists $C >  0$  such that for every $k$ and every $\xi \in {\mathcal H},$ one has
	$$\|T^k\xi\| \leq  C\|\xi\|.$$
In \cite[Theorem 2]{COT}, it was proved that every power bounded $m$-isometry operator is an isometry. The following theorem extend this result to $n$-quasi-$m$-isometry.

\begin{theorem}\label{2.5}

If $T\in {\mathcal L}({\mathcal H})$ is an $n$-quasi-$m$-isometric operator which is power bounded, Then $T$ is an $n$-quasi-isometry.
\end{theorem}
\begin{proof}
We consider the
following two cases:\par \vskip 0.2 cm \noindent Case 1: If $\overline{{\mathcal R}(T^n)}$ is dense, then $T$ is an $m$-isometric operator which is power bounded, thus $T$ is a isometry by \cite[Theorem 2]{COT}. It follows  that $T$ is an $n$-quasi-isometry. \par \vskip 0.2 cm \noindent Case 2: If $\overline{{\mathcal R}(T^n)}$ is not dense. By Theorem \ref{2.1}  we write the matrix representation of $T$ on
${\mathcal H}=
\overline{{\mathcal R}(T^n)} \oplus {mathcal N}(T^{*n})$ as follows $T=\left(
         \begin{array}{ccc}
           T_1 & T_2 \\
          0 & T_3
         \end{array}
       \right)$
where $T_1 = T|\overline{{\mathcal R}(T^n)}$ is an $m$-isometric operator
and $T_3^n=0.$ By taking into account that $T$ is power bounded,  it is easily to check
that $T_1$ is power bounded. From which we deduce that $T_1$ is an isometry.  The result  follows by applying the statement $(2)$ of Theorem \ref{2.1}
\end{proof}

\par \vskip 0.2 cm \noindent Recall that for two operators $T, S$
in ${\mathcal L}({\mathcal H})$, the commutator $[T, S]$ is defined to be
$$[T, S] = TS-ST.$$
\par \vskip 0.2 cm \noindent  A  pair of operators $(T, S)\in {\mathcal L}({\mathcal H})^2$ is said
to be a doubly commuting pair if $(T, S)$ satisfies $TS = ST$ and $T^*S=ST^*$ or equivalently $[T,S]=[T,S^*]=0.$
\par\vskip 0.2 cm \noindent
In \cite[ Theorem 2.2]{OS}, it has proved that if $T$ and $S$ are commuting bounded
linear operators on a Banach space such that $T$ is a $2$-isometry and $S$ is an
$m$-isometry, then $ST$ is an $(m+1)$-isometry. This result was improved in \cite[Theorem 3.3]{BMN} as follows: if $T S = ST$, $T$ is an $(m, p)$-isometry and $S$ is an $(l, p)$-isometry,
then $ST$ is an $(m +l-1, p)$-isometry.
It is natural to ask whether the product  of two $n$-quasi-$m$-isometries is $n$-quasi-$m$-isometry.
The following theorem gives an affirmative answer  under suitable conditions.\par \vskip 0.2 cm \noindent
\begin{theorem}\label{th26} Let $S$ and $T$ be in ${\mathcal L}({\mathcal H})$ are doubly commuting operators and let $m,l,n_1,n_2$ be positive integers.
If $T$ is an $n_1$-quasi-$m$-isometry  and $S$ is an $n _2$-quasi-$l$-isometry, then   $TS$ is a $n_0=\max\{n_1,n_2\}$-quasi-$(m+l-1)$-isometry.
\end{theorem}

\begin{proof} Under the assumption that $T$ and $S$ are doubly commuting, it follows that
$[T^*, S^*]=[T,S]=[T,S^*]=0.$ By taking into account  \cite[Lemma 12]{CG} we obtain that
\begin{eqnarray*}
&&\beta_{m+l-1,\;n_0}(TS)\\&=&\big(TS\big)^{*(n_0)}\beta_{m+l-1}(TS)\big(TS\big)^{n_0}\\
&=&(T^*)^{n_0}(S^*)^{n_0} \, \Lambda_{m+l-1}(TS) \, (T)^{n_0}(S)^{n_0}\\&=&
 (T^*)^{n_0}(S^*)^{n_0} \bigg(\sum_{0\leq j \leq m+l-1} \binom{m+l-1}{j}  {T^{\ast}}^j\beta_{m+l-1-j}(T) \,T^j  \beta_j(S)  \bigg)(T)^{n_0}(S)^{n_0}
 \cr
&=&  \bigg(\sum_{0\leq j \leq m+l-1} \binom{m+l-1}{j}  \, {T^{\ast}}^j\underbrace{\big(T^*\big)^{n_0} \beta_{m+l-1-j}(T) T^{n_0}}\,T^j  \underbrace{\big(S^*\big)^{ n_0}\beta_j(S) S^{n_0}} \bigg).
\end{eqnarray*}
  Since $S$ is an $n_2$-quasi-$l$-isometry, it follows by Corollary \ref{2.3}  that $\big(S^*\big)^{n_0}\beta_j(S)S^{n_0}=0$  for $j\geq l$. On the other hand, we have if $j\leq l-1$, then $m+l-1-j\geq m+l-1-l+1=m$, and so $ \big(T^*\big)^{n_0}\beta_{k+m-1-j}(T)T^{n_0}=0$ by the fact that $T$ is an $n_1$-quasi-$m$-isometry and Corollary \ref{2.3}. Therefore
 $TS$ is an $n_0=\max\{n_1,n_2\}$-quasi-$(m+l-1)$-isometry.
This completes the proof.
\end{proof}
\par \vskip 0.2 cm \noindent The following example shows that Theorem \ref{th26} is not necessarily true if $S,T$ are
not doubly commuting.\par \vskip 0.2 cm \noindent
\begin{example}
We consider the operators
$ T=\left(
      \begin{array}{cc}
        1 & 1 \\
        0 & 1 \\
      \end{array}
    \right) \ \mbox{and}\ S=\left(
      \begin{array}{cc}
        2 & 1 \\
        -1 & 0 \\
      \end{array}
    \right) $ on the two dimensional Hilbert space $\mathbb{C}^2.$
    Note that $ST\neq TS$. Moreover, by a direct computation, we show that $T$ is a quasi-$3$-isometry and $S$ is a $2$-quasi-$3$-isometry.
    However neither $TS$ nor $ST$ is  a $2$-quasi-$5$-isometry.\end{example}
\par \vskip 0.2 cm \noindent
\begin{corollary}\label{2.5}
Let $T,S\in {\mathcal L}({\mathcal H})$ are doubly commuting operators such that $T$ is an $n_1$-quasi-$m$-isometry and $S$ is an $n_2$-quasi-$l$-isometry, then $T^pS^q$ is a
$\max\{n_1,\;n_2\}$-quasi-$(m+l-1)$-isometry for all positive integers $p$ and $q$.
\end{corollary}
\begin{proof}
Since $T$ and $S$ are doubly commuting, then $T^p$ and $S^q$ are doubly commuting. By Theorem \ref{2.3} we know that $T^p$  is an $n_1$-quasi-$m$-isometry and $S^q$ is an $n_2$-quasi-$l$-isometry. Now  by applying Theorem \ref{th26}, we  get that $T^pS^q$ is a $\max\{n_1,\;n_2\}$-quasi-$(m+l-1)$-isometry.
This completes the proof.
\end{proof}

\par\vskip 0.2 cm \noindent
Let ${\mathcal H} \overline{\otimes} {\mathcal H}$ denote the completion,
endowed with a reasonable uniform cross-norm, of the algebraic tensor product ${\mathcal H} \otimes {\mathcal H}$
of ${\mathcal H}$ and ${\mathcal H}$.  For $ T\in {\mathcal L }({\mathcal H})$  and $S \in {\mathcal L}({\mathcal H})$ , $T \otimes S \in {\mathcal L}({\mathcal H}\overline{\otimes} {\mathcal H})$ denote the tensor product
operator defined by $T$ and $S$. \par \vskip 0.2 cm \noindent In the following proposition we prove that the tensor product of an $n_1$-quasi-$m$-isometric
operator with an $n_2$-quasi-$l$-isometric operator is a $\max\{n_1,n_2\}$-quasi-$(m +l-1)$-isometric operator. This proposition generalizes \cite[Theorem 2.10]{PD}.
\par \vskip 0.2 cm \noindent
\begin{proposition}
If $ T\in {\mathcal L}({\mathcal H})$ is an $n_1$-quasi-$m$-isometry and $S \in {\mathcal L}({\mathcal H})$ is an $n_2$-quasi-$l$-isometry, then $T\otimes  S$ is a $\max\{n_1,n_2\}$-quasi-$(m + l - 1)$-isometric.
\end{proposition}
\begin{proof}
Observe that an operator $T \in {\mathcal L}({\mathcal H})$ is $n$-quasi $m$-isometric if and only if $T \otimes I$ and $I \otimes T $ are $n$-quasi-$m$-isometry.
In view of the fact that $$ T\otimes S = (T \otimes I)(I \otimes S)=(I \otimes S) (T\otimes I)$$ it follows that
$$[T \otimes I, I \otimes S]=[T \otimes I, (I \otimes S)^*]=0.$$ Now $T \otimes I$ is an $n_1$-quasi-$m$-isometry and $I \otimes S$ is an  $n_2$-quasi-$m$-isometry such that $T \otimes I$ and $I \otimes S$ are doubly commuting operators. By applying Theorem \ref{th26} we obtain that
$(T \otimes I)(I \otimes S)$ is a $\max\{n_1,n_2\}$-quasi-$(m +l-1)$-isometric operator. Hence $T\otimes S$ is an $\max\{n_1,n_2\}$-quasi-$(m +l-1)$-isometric as required.
\end{proof} \par \vskip 0.2 cm \noindent
The following corollary is an immediate consequence of Theorem \ref{2.3} and Proposition \ref{2.2}.
We omitted its proof.
\begin{corollary}\label{2.6}
If $ T\in {\mathcal L}({\mathcal H})$ is an $n_1$-quasi-$m$-isometry and $S \in {\mathcal L}({\mathcal H})$ is an $n_2$-quasi-$l$-isometry, then $T^p\otimes  S^q$ is a $\max\{n_1,n_2\}$-quasi-$(m + l - 1)$-isometry.
\end{corollary}
\par\vskip 0.2 cm \noindent
It was proved in \cite[Thoerem 2.2]{BMN1} that if $T\in {\mathcal L}({\mathcal H})$ is
an isometry and $Q\in {\mathcal L}({\mathcal H})$ is a nilpotent operator of order $p$ such
that $TQ=QT$, then $T+Q$-is a strict $(2p-1)$-isometry. Later T. Berm\'{u}dez et al. \cite{TMVN} gave a partial generalization to  $m$-isometry, that is
if $T$ is  an $m$-isometry with $m > 1$, $Q$ is a nilpotent operator with order $p$, and $TQ=QT$, then  $T+Q$ in a  $(m+2p-2)$-isometry.
 Recently, C. Gu. and  M. Stankus \cite{CM} gave a  generalization to  $m$-isometry, that is
if $T$ is  an $m$-isometry with $m > 1$, $Q$ is a nilpotent operator with order $p$, and $TQ=QT$, then  $T+Q$ is a  strict $(m+2p-2)$-isometry.
 The following Theorem  states the corresponding partial generalization
 to the sum of an $n$-quasi-$m$-isometry and a nilpotent operator.
\par \vskip 0.2 cm \noindent
\begin{theorem}\label{2.7}
Let $ T, Q \in {\mathcal L}({\mathcal H})$ such that $T$ commutes with $Q$. If $T$  is an $n$-quasi-$m$-isometry and $Q$
is a nilpotent operator of order $p$, then $T +Q$ is a $(n+p)$-quasi-$(m +2p-2)$-isometry.
\end{theorem}

\begin{proof} We need to show that $\beta_{m+2p-2,\;\alpha}\big(T+Q\big)=0.$ Set  $q=m+2p-2$ and $\alpha=n+p$,
by \cite[Lemma 1]{CM} we have
$$\beta_q\big(T+Q\big)=\sum_{0\leq k\leq q}\sum_{0\leq j\leq q-k}\binom{q}{k}\binom{q-k}{j}\big(T^*+Q^*\big)^kQ^{*j}\beta_{q-k-j}(T)T^jQ^k.$$

In fact, note that
\begin{eqnarray*}
&&\beta_{m+2p-2,\;\alpha}\big(T+Q\big)\\&=&\big(T+Q\big)^{*\alpha}\beta_{m+2p-2}\big(T+Q\big)\big(T+Q\big)^{\alpha}\\&=&
\bigg(\sum_{0\leq r\leq \alpha}\binom{\alpha}{r}T^{*(\alpha-r)}Q^{*r}\bigg)\bigg(\sum_{0\leq k\leq q}\sum_{0\leq j\leq q-k}\binom{q}{k}\binom{q-k}{j}\big(T^*+Q^*\big)^kQ^{*j}\beta_{q-k-j}(T)T^jQ^k \bigg)\\&&\times \bigg(\sum_{0\leq r\leq \alpha}\binom{\alpha}{r}T^{\alpha-r}Q^{r}\bigg).
\end{eqnarray*}
Now observe that if $k\geq p$ or $j\geq p$ then $Q^{*k}=0$ or $Q^{*j}=0$ and hence $$\big(T^*+Q^*\big)^kQ^{*j}\beta_{q-k-j}(T)T^jQ^k=0.$$
\noindent However , if $k< p$ and $j< p$, we obtain
$$q-k-j=m+2p-2-k-j\geq m+2p-2-(p-1)-(p-1)=m$$  and using the fact that $T$ is an $n$-quasi-$m$-isometry, we get
$$T^{*(n+p-r)}\beta_{q-k-j}(T)T^{n+p-r}=0\;\;\;\text{for}\;\;r=0,\cdots,p$$  and
$$T^{*(n+p-r)}Q^{*r}\beta_{q-k-j}(T)T^{n+p-r}Q^r=0\;\;\;\text{for}\;\;r=p+1,\cdots,n+p.$$
Combining the above arguments  we obtain
$\beta_{m+2p-2,\;n+p}\big(T+Q\big)=0.$
\end{proof}
\par \vskip 0.2 cm \noindent
\begin{remark}
A simple example shows that the commuting condition of $T$ and $Q$ can not
be removed from the above theorem.\end{remark}
\begin{example}
Let $T=\left(
         \begin{array}{cc}
           -5 & 0 \\
           0 & -1 \\
         \end{array}
       \right)$ and $Q=\left(
                         \begin{array}{cc}
                           0 & 1 \\
                           0 & 0 \\
                         \end{array}
                       \right)
       .$ Then T is a quasi-$3$-isometry and $Q^2 = 0$. Set $S=T+Q$, by direct calculation we show that $T$ is not $5$-quasi-$5$-isometry.
\end{example}
\par \vskip 0.2 cm \noindent
\begin{corollary}\label{2.6}
Let $T \in {\mathcal L}({\mathcal H})$ be an $n$-quasi-$m$-isometry and $ Q\in {\mathcal L}({\mathcal H})$ be a nilpotent operator of order $p$. Then $T\otimes I+I\otimes Q$ is a $(n+ p)$-quasi-$(m+2p-2)$-isometry.
\end{corollary}
\begin{proof}
We note that $ T\otimes I\in {\mathcal L}({\mathcal H} \overline{\otimes} {\mathcal H})$ is an $n$-quasi-$m$-isometry and  $   I\otimes Q\in {\mathcal L}({\mathcal H} \overline{\otimes} {\mathcal H})$ is a nilpotent of order $p$. Moreover $(T\otimes I)(I\otimes Q)=(I\otimes Q)(T\otimes I)$.
\end{proof}
\par \vskip 0.2 cm \noindent

\begin{corollary} Let $T_j \in {\mathcal L}({\mathcal H})$ be an $n_j$-quasi-$m_j$-isometry for $j=1,\cdots ,d$. Define
$$S=\left(
    \begin{array}{cccc}
      T_1 & \alpha_1 I & 0 &\cdots \\
      0 &  \ddots &  \ddots &\ddots \\
      \ddots &   \ddots &  \ddots  & \alpha_{d-1}I \\
      0 &  \ddots  & 0& T_d \\
    \end{array}
  \right) \;\text{on}\;\;{\mathcal H}^{(d)}:={\mathcal H}\oplus\cdots \oplus{\mathcal H}
$$
  where $\alpha_j$ is a complex number for each $j = 1,\cdots, d-1 $ and ${\mathcal H}^{(d)}$ is the sum of $d$-copies of ${\mathcal H}$, then ${ S}$ is a
 $(n+d)$-quasi-$(m+2d-2)$-isometry  where $n=\displaystyle\max_{1\leq j\leq d}\{n_j,\}$ and $m=\displaystyle\max_{1\leq j\leq d}\{m_j,\}.$
\end{corollary}

\begin{proof} Consider the two  operators

$$T=\left(
    \begin{array}{cccc}
      T_1& 0 & 0 &\cdots \\
      0 &  \ddots &  \ddots &\ddots \\
      \ddots &   \ddots &  \ddots  & 0 \\
      0 &  \ddots  & 0& T_d\\
    \end{array}
  \right)
\;\;\text{and}\;\;
Q=\left(
    \begin{array}{cccc}
      0& \alpha_1I & 0 &\cdots \\
      0 &  \ddots &  \ddots &\ddots \\
      \ddots &   \ddots &  \ddots  & \alpha_{d-1}I \\
      0 &  \ddots  & 0& 0\\
    \end{array}
  \right)\;\;\text{for}\;\;j=1,\cdots,d.$$ It is easy to see that
  ${ S}={ T}+{ Q}$. Observing that ${ T}$ is an $n$-quasi $m$-isometry tuple (by Corollary \ref{2.3}) where
  $n=\displaystyle\max_{1\leq j\leq d}\{n_j\}$ and $m=\displaystyle\max_{1\leq j\leq d}\{m_j\}$.
   Therefore $ Q $ is a $d$-nilpotent operator. According to Theorem \ref{2.7}, ${S}$ is a $(n+d)$-quasi-$(m+2d-2)$-isometry..
\end{proof}

\par \vskip 0.2 cm \noindent

The following theorem shows that the
class of $n$-quasi-$m$-isometry  is a closed subset of
${\mathcal L}({\mathcal H})$ equipped with the uniform operator (norm) topology.
\par \vskip 0.2 cm \noindent
\begin{theorem}\label{2.8} Let  $T \in {\mathcal L}({\mathcal H}).$
If $(T_k)_k $ is a sequence of $n$-quasi-$m$-isometry such that
$\displaystyle\lim_{k \to \infty }\|T_k- T\|= 0,$  then $T $ is also $n$-quasi-$m$-isometry.
\end{theorem}
\begin{proof}
Suppose that $(T_k)_k$ is a sequence of $n$-quasi-$m$-isometric operators such
that $$\displaystyle\lim_{k \to \infty}\|T_k -T\| = 0.$$  Since for every positive integer $k$. $T_k$ is an $n$-quasi-$m$-isometry, we have $\beta_{m,n}(T)=0$. It follows that
\begin{eqnarray}\label{norm}
&&\|\beta_{m,\;n}(T)\|=\|\beta_{m,\;n}(T_k)-\beta_{m,\;n}(T)\|\nonumber\\&=&\|T_k^{*n}\bigg(\sum_{0\leq j \leq m}(-1)^{m-j}\binom{m}{j}T_k^{*j}T_k^j\bigg)T_k^n-T^{*n}\bigg(\sum_{0\leq j \leq m}(-1)^{m-j}\binom{m}{j}T^{*j}T^j\bigg)T^n\|\nonumber\\
&\leq& \|T_k^{*n}\bigg(\sum_{0\leq j \leq m}(-1)^{m-j}\binom{m}{j}T_k^{*j}T_k^j\bigg)T_k^n-T_k^{*n}\bigg(\sum_{0\leq j \leq m}(-1)^{m-j}\binom{m}{j}T_k^{*j}T^j\bigg)T^n\|\nonumber\\&+&\|
T_k^{*n}\bigg(\sum_{0\leq j \leq m}(-1)^{m-j}\binom{m}{j}T_k^{*j}T^j\bigg)T^n-T^{*n}\bigg(\sum_{0\leq j \leq m}(-1)^{m-j}\binom{m}{j}T^{*j}T^j\bigg)T^n\|
\nonumber \\&\leq&\|\sum_{0\leq j \leq m}(-1)^{m-j}\binom{m}{j}T_k^{*n+j}\bigg(T_k^jT_k^n-T^jT^n\bigg)\|\nonumber
\end{eqnarray}
\begin{eqnarray}
&&+\|\sum_{0\leq j \leq m}(-1)^{m-j}\bigg(T_k^{*n+j}-T^{*n+j}\bigg)T^jT^n\|\nonumber\\&\leq&\sum_{0\leq j \leq m}\binom{m}{j}\|T_k^*\|^{n+j}\|T_k^jT_k^n-T^jT^n\|+\sum_{0\leq j \leq m}\binom{m}{k}\|T_k^{*n+j}-T^{n+j}\|\|T\|^{j+n}. \ \ \ \ \ \
\end{eqnarray}
Since the product of operators is sequentially continuous in the strong topology, one concludes that $T_k^jT_k^n$, $T_k^{n+j}$
converge strongly to $T^jT^n$  and $T^{n+j}$  respectively for $j=0,1,...,m$. Hence the limiting case of $(\ref{norm})$ shows that $T$ belongs to the class of $n$-quasi-$m$-isometric operators.
\end{proof}
\medskip
%We shall now consider the property of invariance with respect to conformal
%automorphisms of the unit disc.

%\begin{lemma}(\cite{OA})
%Let $n\geq 1$ be an integer, and let $T \in
%\mathcal{B}(\mathcal{H})$ an operator such that $r(T)\leq 1$. Then
%the following equality hold
%\begin{eqnarray*}
%&&\sum_{0\leq k\leq
%n}\binom{m}{k}\varphi_\alpha(T)^{*k}\varphi_\alpha(T)^k\\&=&(1-|\alpha|^2)^m\big(I-\alpha
%T^*\big)^{-m}\big(\sum_{0\leq k\leq
%n}(-1)^k\binom{m}{k}T^{*k}T^k\bigg)\big(I-\overline{\alpha}T\big)^{-m}
%\end{eqnarray*}
%holds for every conformal automorphism $\varphi_\alpha$ of the unit
%disc of the form
%$\varphi_\alpha(z)=\displaystyle\frac{z-\alpha}{1-\overline{\alpha}z}$
%for all $z \in \mathbb{D}$ and $\alpha \in \mathbb{D}.$
%\end{lemma}
%\par \vskip 0.2 cm \noindent
%Let $Aut(\mathbb{D})$ be the group of all conformal mapping from
%$\mathbb{D}$ onto itself ( also called disk automorphisms of
%$\mathbb{D}$). It is well known that $Aut(\mathbb{D})$ coincides
%5with the set of all M\"{o}bius transformations of $\mathbb{D}$ onto
%itself:
%$$5
%Aut(\mathbb{D})= \{\lambda \varphi_{\alpha} : |\lambda|=1, \alpha
%\in \mathbb{D} \}.
%$$
% \vspace {-2mm}
%\begin{remark} It was observed in \cite{OS} that the conformal automorphisms operate on the
%class of $m$-isometries.\end{remark}
% \vspace {-2mm}
%\begin{proposition}
%If $T \in \mathcal{L}(\mathcal{H})$ is an $n$-quasi-$m$-isometry, then so is
%$\varphi(T)$ for every $\varphi \in Aut(\mathbb{D}).$
%\end{proposition}
%\begin{proof}
%\end{proof}

\medskip
\section{$n$-quasi strict-$m$-isometries}
In this section  we introduce and study some properties of the class of $n$-quasi strict- $m$-isometric operators.\par \vskip 0.2 cm \noindent
Recall that an operator
 $T \in {\mathcal L}({\mathcal H})$ is said to be a strict $m$-isometry if $T$ is an $m$-isometry but it is not an $(m-1)$-isometry.
 \par\vskip 0.2 cm \noindent
\begin{definition}\label{3.1}
We say that $T \in {\mathcal L}({\mathcal H})$ is a  $n$-quasi strict $m$-isometry if $T$ is an $n$-quasi-$m$-isometry, but
$T$ is not an $n$-quasi-$(m -1$)-isometry.
\end{definition}
\par \vskip 0.2 cm \noindent
\begin{example}(\cite{SP}) Let $\big(e_k\big)_{k\geq 1}$
be an orthonormal basis of a Hilbert space ${\mathcal H}$. Consider
an  operator $T \in {\mathcal L}({\mathcal H})$ defined by: $$\left\{
                     \begin{array}{ll}
                      Te_1=ae_2& a\neq \sqrt{2}\\
                      Te_p= \sqrt{\frac{p+1}{p}}e_{p+1}& (p=2,3,...)
                     \end{array}
                   \right.$$
 A direct calculation shows that  $T$ is  a quasi-$2$-isometric, but not a quasi-isometric. Therefore $T$ is a quasi strict-$2$-isometry.
 \end{example}
\par \vskip 0.2 cm \noindent
 \begin{example}
 Consider the operator $T\in {\mathcal L}(\mathbb{C}^2)$ given by
 $T=\left(
 \begin{array}{cc}
        1 & 1 \\
        0 & 1 \\
      \end{array}
    \right) $
    who is quasi-$3$-isometric but is not quasi-$2$-isometric. Hence $T$ is a quasi strict-$3$-isometry.
\end{example}
\par \vskip 0.2 cm \noindent
\begin{remark}
It is proved in Corollary \ref{2.3} that an $n$-quasi-$m$-isometric operator is $n$-quasi-$k$-isometric operator
for all integers $k\geq  m.$ Hence if an $T \in {\mathcal L}({\mathcal H})$ is a strict $n$-quasi-$m$-isometry, then it is not a $n$-quasi-$k$-isometry for all integers
$1 \leq  k < m.$
\par \vskip 0.2 cm \noindent
\end{remark}
\par \vskip 0.2 cm \noindent Recall that the multinomial coefficients  is given by $\displaystyle\binom{k}{p_1,\cdots,p_k}=\frac{k!}{p_1! p_2!\cdots p_k!}$ where $k$ and $p_1,\cdots,p_k$ are nonnegative integers subject to $k = p_1+p_2+\cdots +p_k$.\par \vskip 0.2 cm \noindent We will use  the following formula for commuting variables
$z = (z_1, . . . , z_q):$  $$\big(z_1+\cdots+z_q)^k=\displaystyle\sum_{p_1+p_2+\cdots+p_q=k}\binom{k}{p_1,p_2,\cdots,p_q}z_1^{p_1}z_2^{p_2}.\cdots z_q^{p_q}.$$ In particular, if $z_1 = \cdots� = z_q = 1,$ we have
$$\sum_{p_1+p_2+\cdots+p_q=k}\binom{k}{p_1,p_2,\cdots,p_q}=q^k.$$\par \vskip 0.2 cm \noindent
\begin{proposition}\label{3.1}
Let $T\in {\mathcal L}({\mathcal H})$, the following statements hold. \par \vskip 0.2 cm \noindent $(1)$  If $m$ is a positive integer and $x\in {\mathcal H}$, then
\begin{equation}\label{3.1}
\Delta_{m,\;n}(T,x)=\Delta_{m-1,\;n}(T, Tx)-\Delta_{m-1,\;n}(T,x).
\end{equation}
In particular, if $T$ is an $n$-quasi-$m$-isometry, then  for every positive integer $k$ one has
\begin{equation}\label{3.2}
\Delta_{m-1,\;n}(T,T^kx)=\Delta_{m-1,\;n}(T,x).\end{equation}

 \par \vskip 0.2 cm \noindent $(2)$ If $k$ is a positive integer and $x\in {\mathcal H}$, then
\begin{equation}\label{3.3}
\Delta_{m,\;n}(T^k,x)=\sum_{p_1+\cdots+p_k=m}\binom{m}{p_1,\cdots,p_k}\Delta_{m,\;n}\big(T,T^{(0.p_1+1.p_2+\cdots+(k-1)p_k)+(k-1)n}x\big).
\end{equation}
\end{proposition}
\begin{proof} For $m\geq 1$ and $x\in {\mathcal H}$ we have
\begin{eqnarray*}
\Delta_{m,\;n}(T,x)&=& \sum_{0\leq k\leq m}(-1)^{m-k}\binom{m}{k}\|T^{k+n}x\|^2\\&=&
(-1)^m\|T^nx\|^2+\sum_{1\leq k\leq m}(-1)^{m-1-k}\binom{m}{k}\|T^{k+n}x\|^2+\|T^{m+n}x\|^2
\\&=&
(-1)^m\|T^nx\|^2+\sum_{1\leq k\leq m}(-1)^{m-1-k}\bigg(\binom{m-1}{k}+\binom{m-1}{k-1}\bigg)\|T^{k+n}x\|^2+\|T^{m+n}x\|^2\\&=&
\sum_{0\leq k\leq m-1}(-1)^{m-1-k}\binom{m-1}{k}\|T^{k+n}Tx\|^2-\sum_{0\leq k\leq m-1}(-1)^{m-1-k}\binom{m-1}{k}\|T^{k+n}x\|^2\\&=&\Delta_{m-1,\;n}(T, Tx)-\Delta_{m-1,n}(T,x).
\end{eqnarray*}
Hence $(\ref{3.1})$ is proved.\par \vskip 0.2 cm \noindent  If we assume that  $T$ is an $n$-quasi-$m$-isometry, then  $\Delta_{m,\;n}(T,x)=0$ and so that
$$\Delta_{m-1,\;n}(T,Tx)=\Delta_{m-1,\;n}(T,x).$$
Hence, $$\Delta_{m-1,\;n}(T,T^kx)=\Delta_{m-1,\;n}(T,T^{k-1}x)=\cdots =\Delta_{m-1,\;n}(T,x).$$
\par\vskip 0.2 cm \noindent $(2)$  From  $(2.3)$ we have
\begin{equation}
\Delta_{m,\;n}(T^k,x)=\sum_{0\leq j \leq m}(-1)^{m-j}\binom{m}{j}\|\big(T^k\big)^{j+n}x\|^2
\end{equation}
and
\begin{eqnarray*}
&&\sum_{p_1+\cdots+p_k=m}\binom{m}{p_1,\cdots,p_k}\Delta_{m,\;n}\big(T,T^{(0.p_1+1.p_2+\cdots+(k-1)p_k)+(k-1)n}x\big)\\&=&
\sum_{p_1+\cdots+p_k=m}\binom{m}{p_1,\cdots,p_k}\bigg(\sum_{0\leq j\leq m}(-1)^{m-j}\binom{m}{j}\|T^{j+kn}T^{(0.p_1+1.p_2+\cdots+(k-1)p_k)}x\|^2\bigg).
\end{eqnarray*}
Thus we need to prove the following
combinatorial identity:

\begin{eqnarray*}&&\sum_{0\leq j\leq m}(-1)^{m-j}\binom{m}{j}(z^k)^{j+n}\\&=&\sum_{p_1+\cdots+p_k=m}\binom{m}{p_1,\cdots,p_k}\bigg(\sum_{0\leq j\leq m}(-1)^{m-j}\binom{m}{j}z^{j+kn}z^{(0.p_1+1.p_2+\cdots+(k-1)p_k)}\bigg).\end{eqnarray*}
In fact, observe that
$$\sum_{0\leq j\leq m}(-1)^{m-j}\binom{m}{j}(z^k)^{j+n}=\big(z^k-1\big)^mz^{kn}$$ and
\begin{eqnarray*}&&\sum_{p_1+\cdots+p_k=m}\binom{m}{p_1,\cdots,p_k}\bigg(\sum_{0\leq j\leq m}(-1)^{m-j}\binom{m}{j}z^{j+kn}z^{(0.p_1+1.p_2+\cdots+(k-1)p_k)}\bigg)\\&=&
z^{kn}\sum_{p_1+\cdots+p_k=m}\binom{m}{p_1,\cdots,p_k}\bigg(\sum_{0\leq j\leq m}(-1)^{m-j}\binom{m}{j}z^{j}\bigg)z^{(0.p_1+1.p_2+\cdots+(k-1)p_k)}.
\end{eqnarray*}
By applying
the multinomial formula in reverse order, we have
\begin{eqnarray*}
&&\sum_{p_1+\cdots+p_k=m}\binom{m}{p_1,\cdots,p_k}\bigg(\sum_{0\leq j\leq m}(-1)^{m-j}\binom{m}{j}z^{j}\bigg)z^{(0.p_1+1.p_2+\cdots+(k-1)p_k)}\\&=&
\big(z-1\big)^m\sum_{p_1+\cdots+p_k=m}\binom{m}{p_1,\cdots,p_k}z^{(0.p_1+1.p_2+\cdots+(k-1)p_k)}\\&=&
\big(z-1\big)^m\big(z^{k-1}+z^{k-2}+\cdots+z+1)^m=\big(z^k-1)^m.
\end{eqnarray*}
\end{proof}

\begin{theorem}\label{3.1}
If $T \in {\mathcal L}({\mathcal H})$  is a $n$-quasi strict $m$-isometry, then for any positive integer $k$, $T^k$ is a $n$-quasi strict $m$-isometry. Furthermore
$$\Delta_{m-1,\;n}(T^k,x)=k^{m-1}\Delta_{m-1,\;n}(T,x).$$
\end{theorem}
\begin{proof}
Since $T$ is a $n$-quasi strict $m$-isometry, by Theorem \ref{2.3}, $T^k$ is an $n$-quasi-$m$-isometry. Furthermore, by (\ref{3.2}) and (\ref{3.3})   we get
\begin{eqnarray*}
\Delta_{m-1,n}(T^k,x)&=&\sum_{p_1+\cdots+p_k=m-1}\binom{m-1}{p_1,\cdots,p_k}\Delta_{m-1,n}(T, T^{(0.p_1+1.p_2+\cdots+(k-1)p_k)+(k-1)n}x)\\&=&
\sum_{p_1+\cdots+p_k=m-1}\binom{m-1}{p_1,\cdots,p_k}\Delta_{m-1,n}(T,x)\\&=&k^{m-1}\Delta_{m-1,\;n}(T,x).
\end{eqnarray*}
Consequently, $T^k$ is a quasi strict-$m$-isometry.
\end{proof}
\par \vskip 0.2 cm \noindent
\begin{remark}
The converse of the above theorem is not true in general. In fact, by  Theorem \ref{2.4}  if  $T^r$ and $T^s$ are $n$-quasi-$m$-isometries for two coprime
positive integers $r$ and $s$, then $T$ is an  $n$-quasi-$m$-isometry.
\end{remark}
Recall that a  sequence $(a_j)_{j\geq 0}$  in a group $G$ is an
arithmetic progression of order $h$ if
$$\sum_{0\leq k \leq h+1}(-1)^{h+1-k}a_{k+j}=0$$ for any $j\geq0.$ An arithmetic progression of order $h$ is of strict order $h$ if $h = 0$ or if $h \geq 1$ and it is not of
order $h -1.$  We refer the interested reader to \cite{TBMN}  for complete details.
\begin{lemma} \label{3.1} (\cite[Theorem 4.1]{TBMN})
 Let $a = (a_j)_{j\geq 0}$ be a numerical sequence. Suppose that $(a_{cj})_{j\geq }$ is an arithmetic progression of strict order $h$ and $(a_{dj})_{j\geq 0}$ is an arithmetic
progressions of strict order $k \geq 0$, for $ c, d \geq  1$ and $h, k \geq 0.$ Then $(a_{ej})_{j\geq 0}$ is
an arithmetic progression of strict order $l$, being $e$ the greatest common divisor
of $c$ and $d$, and $l$ the minimum of $h$ and $k$.
 \end{lemma}
\begin{theorem}\label{3.2}
Let $T\in {\mathcal L}({\mathcal H})$ and $r, s,n,m,l$ be positive integers. If $T^r$
is a $n$-quasi strict $m$-isometry and $T^s$
is an $n$-quasi strict $l$-isometry, then $T^q$
is an $n$-quasi-$p$-isometry,
where $q$ is the greatest common divisor of $r$ and $s$, and $p$ is the minimum of $m$ and $l.$
\end{theorem}
\begin{proof}
Fix $x\in {\mathcal H}$ and set $a_j:=\|T^jx\|^2$ for $j=1,2,\cdots$. Since $T^r$ is a $n$-quasi strict $m$-isometry, it follows that  $(a_{r(j+n)})_{j\geq 0}$ is an arithmetic progression of strict order $m-1$ satisfies the
recursive equation
$$\sum_{0\leq k\leq m}(-1)^{m-k}\binom{m}{k}a_{r(k+n)+j}=0 \;\;\quad \text{for all}\;\;j\geq 0.$$
Analogously, as $T^s$ is a $n$-quasi strict-$l$-isometry, it follows that $(a_{s{(j+n)}})_{j\geq 0}$  is an arithmetic progression of strict order $l-1$ satisfies the recursive equation
$$\sum_{0\leq k\leq l}(-1)^{l-k}\binom{l}{k}a_{s(k+n)+j}=0\;\;\quad \text{for all}\;\;j\geq 0.$$
 Applying Lemma \ref{3.1} it results that
$(a_{q{(j+n)}})_{j\geq 0}$ is an arithmetic progression of strict order $p-1$, so $T^q$
is an $n$-quasi-strict $p$-isometry, where $q = gcd(r,s )$ and $p =\min\{m,l\}$.
\end{proof}
The following corollary shows that
if we assume that two suitable different powers of $T$ are $n$-quasi strict-$m$-isometries,
then we obtain that $T$ is a $n$-quasi strict-$m$-isometry.
\par\vskip 0.2 cm \noindent
\begin{corollary}\label{2.4}
Let $T \in {\mathcal L}({\mathcal H})$ and  $r, s,m,n,l$ be
positive integers. The following properties hold. \par \vskip 0.2 cm \noindent $(1)$
If $T$ is an $n$-quasi strict-$m$-isometry such that $T^s$ is an $n$-quasi strict-isometry,

then $T$ is an $n$-quasi strict-isometry.\par \vskip 0.2 cm \noindent $(2)$
 If $T^r$ and $T^{r+1}$ are $n$-quasi strict-$m$-isometries, then so is $T$.\par \vskip 0.2 cm \noindent $(3)$
 If $T^r$ is an $n$-quasi strict-$m$-isometry and $T^{r+1}$ is an $n$-quasi strict-$l$-isometry

  with
$m < l$, then $T$ is an $n$-quasi strict-$m$-isometry.
\end{corollary}
\begin{proof}

 The proof is an immediate consequence of Theorem \ref{3.2}.
\end{proof}
\begin{theorem}\label{3.3}
Let $T\in {\mathcal L}({\mathcal H})$ and $S\in {\mathcal L}({\mathcal H})$ be doubly commuting operators. If $T$ is an $n$-quasi strict $m$-isometry and $S$ is an $n$-quasi strict-$l$-isometry, then $TS$ is a $n$-quasi strict-$(m+l-1)$-isometry if and only if $$\big(T^*\big)^{n+l-1}\beta_{m-1}(T)T^{n+l-1}S^{*n}\beta_{l-1}(S)S^{n}\not=0.$$
\end{theorem}
\begin{proof} In view of  Theorem \ref{2.6}, it is obvious that $\beta_{m+l-1,\;n}(TS)=0$. On the other hand,
since $T$ and $S$ are doubly commuting operators, it follows by \cite[Corollary 3.9]{CM} that
$$\beta_q(TS)=\sum_{0\leq k\leq q}\binom{q}{k}T^{*k}\beta_{q-k}(T)T^k\beta_k(S).$$
From which it follows that
\begin{eqnarray*}
&&\beta_{m+l-2,\;n}(TS)\\&=&\big(TS)^{n}\beta_{m+l-2}(TS)\big(TS\big)^{n}\\&=&
\big(T^*\big)^{n}\big(S^*\big)^{n}\bigg( \sum_{0\leq k\leq m+l-2}\binom{m+l-2}{k}T^{*k}\beta_{m+l-2-k}(T)T^k\beta_k(S)\bigg)T^{n}S^{n}\\&=&
\sum_{0\leq k\leq m+l-2}\binom{m+l-2}{k}T^{*k}\big(T^*\big)^{n}\beta_{m+l-2-k}T^{n}T^k\big(S^*\big)^{n}\beta_k(S)S^{n}\\&=&
\sum_{0\leq k\leq l-2}\binom{m+l-2}{k}T^{*k}\big(T^*\big)^{n}\beta_{m+l-2-k}T^{n}T^k\big(S^*\big)^{n}\beta_k(S)S^{n}\\&&+\binom{m+l-2}{l-1}\big(T^*\big)^{n+l-l}\beta_{m-1}(T)T^{n+l-1}\big(S^*\big)^{n}\beta_{l-1}S^n
\\&&+\sum_{l\leq k\leq m+l-2}\binom{m+l-2}{k}T^{*k}\big(T^*\big)^{n}\beta_{m+l-2-k}T^{n}T^k\big(S^*\big)^{n}\beta_k(S)S^{n}.
\end{eqnarray*}
If $k\in [0,l-2]$, then $(m+l-2-k)\in [m,m+l-2]$  and hence $T^{*n}\beta_{m+l-2-k}(T)T^n=0$ by Corollary \ref{2.3}.  If $k\geq l$, then $S^{n}\beta_k(S)S^n=0$ also in view of Corollary \ref{2.3}.\par \vskip 0.2 cm \noindent Consequently, $TS$ is a $n$-quasi strict-$(m+l-1)$-isometry if and only if
$$\big(T^*\big)^{n+l-l}\beta_{m-1}(T)T^{n+l-1}\big(S^*\big)^{n}\beta_{l-1}(S)S^n=0.$$  Hence the proof is finished.
\end{proof}
\begin{theorem}
Let $ T\in {\mathcal L}({\mathcal H})$ and  $S \in  {\mathcal L}({\mathcal H})$. If $ T$ is a $n$-quasi strict-$m$-isometry
and $S$ is a $n$-quasi strict-$l$-isometry, then $T \otimes S$ on ${\mathcal H} \overline{\otimes} {\mathcal H}$ is a  $n$-quasi strict-$(m+l-$1)-isometry.
\end{theorem}
\begin{proof}
In view of \cite[Corollary 3.10]{CG1}, it follows that
$$\beta_q(T\otimes S)=\sum_{0\leq k\leq q}\binom{q}{k}T^{*k}\beta_{q-k}(T)T^k\otimes\beta_k(S).$$ By calculations we have
\begin{eqnarray*}
&&\beta_{m+l-2,\;n}(T\otimes S)\\&=&\big(T\otimes S\big)^{*n}\beta_{m+l-2}(T\otimes S)\big(T\otimes S\big)^n\\&=&
\big(T^{*n}\otimes S^{*n}\big)\bigg(  \sum_{0\leq k\leq m+l-2}\binom{m+l-2}{k}T^{*k}\beta_{m+l-2-k}(T)T^k\otimes\beta_k(S)\bigg)\big(T^n\otimes S^n\big)\\&=&
 \sum_{0\leq k\leq m+l-2}\binom{m+l-2}{k}T^{*k}T^{*n}\beta_{m+l-2-k}(T)T^nT^k\otimes S^{*n}\beta_k(S)S^n.
\end{eqnarray*}
A similar arguments as in the proof of Theorem \ref{3.3} give
$$\beta_{m+l-2,\;n}(T\otimes S)=\big(T^*\big)^{n+l-l}\beta_{m-1}(T)T^{n+l-1}\otimes \big(S^*\big)^{n}\beta_{l-1}(S)S^n.$$ This
means that $T\otimes S$ is a $n$-quasi strict-$(m+l-1)$-isometry as required.
\end{proof}
\par \vskip 0.2  cm \noindent
In \cite[Theorem 3.1]{BJZ} it has been proved that if $T\in {\mathcal L}({\mathcal H})$ is a strict $m$-isometry, then the list
of operators $\{ \;T^{*k}T^k, k = 0, 1,\;\cdots m - 1 \;\}$ is linearly independent which is equivalent to
that $\{\;\beta_k(T), k = 0, 1,\;\cdots,m-1 \;\}$ is linearly independent.\par \vskip 0.2 cm \noindent In the following proposition we extend this result to $n$-quasi strict-$m$-isometry as follows.
\begin{proposition}
If $T\in {\mathcal L}({\mathcal H})$ is an $n$-quasi strict $m$-isometry,  then the list of operators
$$\{\;\beta_{k,\;n}(T), k = 0, 1,\;\cdots,m-1 \;\}$$ is linearly independent.
\end{proposition}
\begin{proof}  The outline of the proof is inspired from \cite{CGU2}.\par \vskip 0.2 cm \noindent

It was observed in \cite{CGU1} that  $\beta_{k}(T)=T^*\beta_{k-1}(T)T-\beta_{k-1}(T)$  \;\;for all $k\geq 1$.\par \vskip 0.2 cm \noindent
We will
just write \begin{eqnarray*}
\beta_{m,\;n}(T)=T^{*n}\beta_m(T)T^n&=&T^{*n}\bigg(T^*\beta_{m-1}(T)T-\beta_{m-1}(T)\bigg)T^n\\&=&
T^{*n+1}\beta_{m-1}(T)T^{n+1}-T^{*n}\beta_{m-1}(T)T^n.
\end{eqnarray*}
Now assume that for some complex numbers $\lambda_k,$
$$\sum_{0\leq k \leq m-1}\lambda_k\beta_{k,\;n}(T)=0$$ or equivalently
$$\sum_{0\leq k \leq m-1}\lambda_kT^{*n}\beta_kT^n(T)T^n=0.$$ Multiplying the above equation on the left and right by $T^*$ and $T$ and subtracting two
equations, we have
$$\sum_{0\leq k\leq m-1}\lambda_k\bigg(T^{*n+1}\beta_{k}(T)T^{n+1}-T^{*n}\beta_k(T)T^n\bigg)=\sum_{0\leq k \leq m-1}\lambda_k\beta_{k+1,\;n}(T)=0$$
By applying the same procedure to the equation $\displaystyle\sum_{0\leq k \leq m-1}\lambda_k\beta_{k+1,\;n}(T)=0$ we get
$$\sum_{0\leq k \leq m-1}\lambda_k\beta_{k+2,\;n}(T)=0.$$  By continuing this process we obtain

 $$\sum_{0\leq k \leq m-1}\lambda_k\beta_{k+j,\;n}(T)=0\;\;\;\text{for all}\;\;j\in \mathbb{N}.$$
Since that every $n$-quasi-$m$-isometric operator is  $n$-quasi-$k$-isometric operator for all $k\geq m$  (Corollary \ref{cor23} ) we have the following cases:
\par \vskip 0.2 cm \noindent For $j=m-1$,
 $\displaystyle\sum_{0\leq k \leq m-1}\lambda_k\beta_{k+j,\;n}(T)=0 \Rightarrow\lambda_0\beta_{m-1,\;n}(T)=0,$ so $\lambda_0=0.$ \par \vskip 0.2 cm \noindent
 For $j=m-2$, $\displaystyle\sum_{0\leq k \leq m-1}\lambda_k\beta_{k+j,\;n}(T)=0 \Rightarrow\lambda_1\beta_{m-1,\;n}(T)=0,$ so $\lambda_1=0.$ \par \vskip 0.2 cm \noindent Continuing this process we see that all $\lambda_k = 0$ for $k=2,\;\cdots,\;m-1$. Hence the result is proved.
\end{proof}

\end{document}